\documentclass[11pt]{amsart}

\usepackage{xcolor}
\usepackage[all]{xy}

\usepackage{geometry}                
\usepackage[parfill]{parskip}    
\usepackage{graphicx}
\usepackage{amssymb, amsmath, amsthm}
\usepackage{epstopdf}
\usepackage{url}
\usepackage{soul}
\usepackage{calligra}
\usepackage{fancyvrb}
\usepackage{todonotes}
\usepackage[shortlabels]{enumitem}

\usepackage{scalerel}

\usepackage{tikz}\usetikzlibrary{matrix, cd, arrows}
\usepackage{tikz-cd}
\usepackage{xcolor}

\usepackage{tensor}

\usepackage{float}
                                                       
\usepackage[pagebackref]{hyperref}

\usepackage{mathabx,epsfig}

\usepackage{mathscinet}
\usepackage{array}
\usepackage[new]{old-arrows}

\usepackage{thmtools, thm-restate}
\declaretheorem[numberwithin=subsection]{theorem}
\declaretheorem[sibling=theorem]{corollary}

\DeclareMathAlphabet{\mathscr}{T1}{calligra}{m}{n}
\theoremstyle{plain}
\newtheorem{proposition}[theorem]{Proposition}
\newtheorem{lemma}[theorem]{Lemma}
\theoremstyle{definition}

\theoremstyle{remark}
\newtheorem{remark}[theorem]{Remark}

\newcommand{\End}{\mathcal{E}nd}

\DeclareMathOperator{\tr}{tr}

\DeclareMathOperator{\ad}{ad}

\newcommand{\calC}{\mathcal{C}}
\newcommand{\calCt}{\widetilde{\mathcal{C}}}

\newcommand{\LL}{\mathcal{L}}
\newcommand{\MM}{\mathcal{M}}

\newcommand{\liesl}{\mathfrak{sl}}

\newcommand{\GL}{\operatorname{GL}}
\newcommand{\SL}{\operatorname{SL}}

\newcommand{\Id}{\operatorname{Id}}

\newcommand{\prym}{P_0}

\VerbatimFootnotes

\title{Abelianization of the $\SL_2$ Hitchin connection at level four}
\author{Thomas Baier}
\address{Thomas Baier\\ CAMGSD\\
Instituto Superior T\'ecnico\\
Av. Rovisco Pais\\
1049-001 Lisboa\\
Portugal}
\email{thomas.baier@novasbe.pt}
\author[Michele Bolognesi]{Michele Bolognesi}
\address{Michele Bolognesi\\
Universit\'e Grenoble Alpes\\ CNRS\\
IF\\
38000 Grenoble\\ France}
\email{\mbox{michele.bolognesi@univ-grenoble-alpes.fr}}
\author{Johan Martens}
\address{Johan Martens\\ School of Mathematics and Maxwell Institute\\ The University of Edinburgh\\ Peter Guthrie Tait Road\\ Edinburgh EH9 3FD\\ United Kingdom}
\email{johan.martens@ed.ac.uk}
\author{Christian \textsc{Pauly}}
\address{Christian Pauly \\ Laboratoire de Math\'ematiques J.A. Dieudonn\'e \\ UMR  7351 CNRS \\ Universit\'e C\^ote d'Azur
 \\ 06108 Nice Cedex 02, France}
\email{pauly@unice.fr}
\thanks{
MB was supported by the ANR grant ANR-20-CE40-0023. JM was supported in part by EPSRC grant EP/N029828/1.  
}
\date{\today}							
\begin{document}
\begin{abstract}
    We prove that the Hitchin connection for $\SL_2$ at level four can be understood in terms of the Mumford-Welters connections on bundles of abelian theta functions for Prym torsors of all unramified double covers, and use this to show that its monodromy is finite.  This builds on earlier works, for individual curves, of the last named author with Oxbury and Ramanan.  The key ingredients in making this work on the level of connections are equivariant conformal embeddings, and anti-invariant level-rank duality.
\end{abstract}
\maketitle

 \section{Introduction}
This paper is concerned with an application of algebraic geometry to topology.  In particular, it aims to explain a sporadic finiteness phenomenon, first observed by Masbaum \cite{masbaum:1999}, in the so-called quantum representations of mapping class groups of surfaces.

Going back to the genesis of their study in the 19th century, it has been known for a long time that theta functions satisfy a heat equation.  This was elaborated by Welters in \cite{welters:1983}, who showed one can use this idea to build a flat projective connection on the corresponding vector bundles over the moduli space of polarized abelian varieties, compatible with the action of theta groups, introduced by Mumford \cite{mumford:1966}, on the fibers of 
these bundles. 
Welters used this to generalize a result of Andreotti and Mayer on the singular locus of the theta divisor in the moduli space of abelian varieties to positive characteristic.

In \cite{hitchin:1990}, Hitchin similarly established a flat projective connection on the vector bundle $\mathbb{V}(\SL_r,k)$
of non-abelian $\SL_r$-theta functions at level $k$ over a family of smooth projective curves of genus $g \geq 2$, also called the Verlinde bundle 
associated to $\SL_r$ at level $k$. The fiber
of the bundle $\mathbb{V}(\SL_r,k)$ over a curve $\calC$ equals the space $H^0(\mathcal{M}_{\SL_r}, \mathcal{L}_{\SL_r}^k)$
of global sections of the $k$-th power of the determinant line bundle $\mathcal{L}_{\SL_r}$ over the moduli space $\mathcal{M}_{\SL_r}$ of semi-stable rank-$r$ vector bundles with trivial determinant over the curve $\calC$.  Hitchin's motivation came from Chern-Simons theory, as one approach to constructing the Chern-Simons TQFT associates a vector space to each closed oriented surface via geometric quantization of the moduli space of flat unitary connections on the trivial vector bundle over the surface.  This requires the choice of polarization, in this case given by a K\"ahler structure on the moduli space, induced by the choice of a complex structure on the surface.  The Hitchin connection now indicates how the vector space depends on this choice.

Like the Mumford-Welters connection, the Hitchin connection is induced by a \emph{heat operator}, which associates a second-order differential operator on the line bundles over the moduli space of vector bundles over a smooth curve (where both are allowed to vary) to each vector field on the base of a family of smooth projective curves.  Such a heat operator is uniquely determined by its symbol, and Hitchin showed that the dual of the quadratic part of the Hitchin system indeed arises as such a symbol.  

Though the construction in \cite{hitchin:1990} was described in terms of geometric quantization, the Hitchin connection can equally be described purely in terms of algebraic geometry \cite{BBMP:2020}, and also corresponds to the connection constructed out of the representation theory of the Virasoro algebra and an affine Lie algebra, arising from the Wess-Zumino-Witten model in conformal field theory \cite{TUY:1989,tsuchimoto:1993,looijenga:2013}.

Through Mumford's work on the theta group, it can be seen that the monodromy for the Mumford-Welters connection is always finite.  It was also believed that this would be true for the Hitchin connection.  However, in \cite{masbaum:1999} Masbaum studied the quantum representations of the mapping class group for $\SL_2$ arising through skein theory using the Kauffman bracket, and showed that there exists an element in the mapping class group (a product of Dehn twists) whose images have infinite order for all levels apart from levels $1,2,4$ and $8$.  This followed on work by Funar \cite{funar:1999}, who showed that for $\SL_2$ the monodromy group is infinite (again, apart from some sporadic low levels). The construction of Masbaum did not work in the low levels $1, 2, 4, 8$, but later J\o rgensen \cite{egsgaard.jorgensen:2016} also constructed elements of infinite order for level $8$.  Note that Andersen and Ueno showed that these quantum representations are isomorphic to those arising as the monodromy of the Hitchin connection \cite{andersen.ueno:2015}.  (While none of the individual quantum representations are faithful, as Dehn twists get sent to finite order elements, when considering all levels for a fixed group, they do become asymptotically faithful \cite{andersen:2006, freedman.walker.wang:2002}.)

In light of the Grothendieck-Katz conjecture \cite{katz:1972}, which states that the monodromy of a connection that is defined in suitably generality has to be finite if its $p$-curvature vanishes for almost all primes $p$, this result came as a surprise to algebraic geometers. Masbaum's results were confirmed directly in terms of the WZW connection by Laszlo, Sorger, and the last-named author in \cite{LPS:2013}.

In this paper, we will now focus on the sporadic behavior at low levels.  In this context, we recall that level-rank (or strange) duality gives an isomorphism, for families of genus $g$ curves, between the bundle of non-abelian theta functions for $\SL_r$ at level $k$, and the dual of the bundle of theta functions for $\GL_k$ (with determinant fixed at $k(g-1)$),  at level $r$ \cite{marian.oprea:2007,belkale:2008}.  It was shown, moreover, by Belkale that this level-rank duality is flat with respect to the natural flat projective connections on both sides of this isomorphism \cite{belkale:2009}. In particular, when looking at the Hitchin connection for $\SL_2$ at level one, this says it is isomorphic to the dual of the Mumford-Welters connections for the Jacobian varieties at level two, and, like all connections in the abelian case, these have finite monodromy.

In \cite[\S 6.3]{atiyah:1990}, Atiyah conjectured that the Hitchin connections for a family of curves (for any group, and at any level) can be understood in terms of the Mumford-Welters connections for the Prym varieties, fibers of the Hitchin system, of all the smooth spectral covers of the curves.  The status of this general abelianization conjecture at present is unclear \cite{yoshida:2006,teleman:2007}.

Here we will now focus on the particular case of the Hitchin connection for $\SL_2$ at level four, and show that an isomorphism with certain spaces of abelian theta functions, given for individual curves by Oxbury and the last-named author \cite{oxbury.pauly:1996} and further elaborated with Ramanan \cite{pauly.ramanan:2001}, also respects the natural connections on both sides when letting the curves move in families.  This shares some similarities with Atiyah's proposal, but is also distinctly different.  Like Atiyah's abelianization conjecture, this involves covering curves, but unlike the spectral curves (which are always ramified) arising through the Hitchin system, the abelianization we consider involves all (a finite number) unramified covers.

What allows us to understand that the Oxbury-Pauly isomorphism respects the connection is a combination of the fact that the embedding $\liesl_2\hookrightarrow \liesl_3$, given by the adjoint representation, is \emph{conformal}, combined with the flatness of the anti-invariant version of level-rank (or strange) duality, recently developed by the authors in \cite{BBMP:2024}.  In fact, the present work was the motivation for this study of anti-invariant level-rank duality. 

The picture for abelianizations of the Hitchin connection that emerges from this is somewhat different from Atiyah's general conjecture.  In particular, our abelianization is very much sporadic, requiring particular numerical coincidences, not present for general levels. Absent these, the connection exhibits much more complicated behavior than in the abelian case.

We remark that the case of finiteness of the monodromy for the Hitchin connection for $\SL_2$ at level two is still unsettled.  It would also be interesting to compute the $p$-curvature of the Hitchin connection to see if it confirms the predictions of the Grothendieck-Katz conjecture in this setting.

Finally, it is perhaps of interest to mention some relevance of the sporadic low-level behavior.  In particular, it was shown by Freedman, Larsen and Wang in \cite{freedman.larsen.wang:2002} that the topological modular functor coming from the $SU(2)$ Chern-Simons-Witten theory is universal for quantum computation.  In particular, they worked with the fifth root of unity (which corresponds to level three in the geometric approach, the context in which the Hitchin connection arose), avoiding the sporadic behavior that we establish here.

\subsection{Statement of the main results}

Let $\mathcal{M}_g$ denote the moduli stack of smooth complex projective curves of genus $g \geq 2$, and 
$\mathcal{R}_g$ the moduli stack of connected étale double covers $\calCt \rightarrow \calC$ with 
involution $\sigma$. There is a natural forgetful map $\pi_g : \mathcal{R}_g \rightarrow \mathcal{M}_g$.
Over $\mathcal{M}_g$ we consider the Verlinde bundle $\mathbb{V}(\SL_2, 4)$ associated to $\SL_2$ at level $4$, whose fiber over $\calC$ equals the
space of non-abelian theta functions $H^0(\mathcal{M}_{\SL_2}, \mathcal{L}_{\SL_2}^4)$. We consider $\pi_J : J^* \rightarrow \mathcal{M}_g$ the
universal Jacobian torsor parameterizing degree $g-1$ line bundles over smooth genus $g$ curves $\calC$, and denote by  
$\Theta \subset J^*$ the relative canonical theta divisor. We also consider $\pi_P: P^{\rm even} \rightarrow \mathcal{R}_g$ 
the universal
Prym torsor parameterizing line bundles $L$ on $\calCt$ that satisfy $\sigma^* L \cong L^{\vee} \otimes K_{\calCt}$ and such 
that $H^0(\calCt,L)$ is even-dimensional. On $P^{\rm even}$ we denote the relative canonical theta divisor by $\Xi$.

We recall \cite{oxbury.pauly:1996,pauly.ramanan:2001} that for any curve $\calC$ 
there is a canonical isomorphism
\begin{equation} \label{isoC}
 \begin{tikzcd}H^0(J^*_{\calC}, \Theta^3)_+^\vee \oplus 
\bigoplus\limits_{\calCt \rightarrow \calC} H^0(P^{\rm even}_{\calCt/\calC}, \Xi^3)_+^\vee \ar[r,"\sim"] & 
H^0(\mathcal{M}_{\SL_2}, \mathcal{L}_{\SL_2}^4), \end{tikzcd}
\end{equation}
where the last sum is taken over all connected double étale covers $\calCt \rightarrow \calC$ and the subscript $+$
denotes the invariant subspace for the standard involutions on the Jacobian and Prym torsors.

\bigskip

Our main result asserts that the isomorphisms \eqref{isoC} for a general curve $\calC$ globalize to a projectively flat vector bundle isomorphism over $\mathcal{M}_g$.

\begin{restatable*}{theorem}{mainresulta}
\label{mainthm}
Over $\mathcal{M}_g$ we have a projectively flat isomorphism between the two vector bundles
\begin{equation*}\begin{tikzcd}[ampersand replacement=\&] ({\pi_J}_* \Theta^3)^\vee_+ \oplus {\pi_g}_*({\pi_P}_* \Xi^3)_+^\vee  \ar[r,"\sim"]\& \mathbb{V}(\SL_2,4)\end{tikzcd} \end{equation*}
equipped respectively with the 
projective  Mumford-Welters and Hitchin connections. 
\end{restatable*}

In order to establish this result, we also have to check that the projective Mumford-Welters connections add up to a single
flat projective connection on the direct sum ${\pi_J}_* \Theta^3 \oplus {\pi_g}_*({\pi_{P}}_* \Xi^3)$
over $\mathcal{M}_g$ and that this projective connection has finite monodromy (Corollary \ref{combined-con}). As an immediate 
consequence of this, we obtain

\begin{restatable*}{corollary}{mainresultb}\label{maincor}
The $\SL_2$ Hitchin connection at level four has finite monodromy.
\end{restatable*}

\subsection{Organization}
The rest of the paper is organized as follows: in Section \ref{connections} we recall the necessary general background about twisted and projective connections.  In Section \ref{dets} we describe the determinants and ranks of bundles of theta functions that will allow us to switch from projective to twisted connections on bundles on theta functions later on.  Section \ref{flatness} explains how the two main ingredients, (twisted) conformal embeddings and anti-invariant level-rank duality, come into play in this context.  Finally, in Section \ref{finiteness}, we further elaborate on theta structures and put everything together to establish Theorem \ref{mainthm} and Corollary \ref{maincor}.

\subsection{Acknowledgments} The authors would like to thank J\o rgen Ellegaard Andersen, 
Prakash Belkale, Sebastian Casalaina-Martin, Chiara Damiolini, Sorin Dumitrescu, Gregor Masbaum, Swarnava Muk\-ho\-padh\-yay, Zakaria Ouaras, Karim Rega, Graeme Segal, 
Richard Wentworth 
and Hacen Zelaci for useful conversations and remarks at various stages of this work.  

\section{Projective versus twisted connections}
\label{connections}
\subsection{Connections} We collect here the basic results on connections that we will be using throughout the paper.  The basics of this material are well-known (originally going back to \cite{beilinson.kazhdan:1991}), but occur in the literature in many different variants, and we need to introduce a few other elementary corollaries.  We also want to distinguish (in our terminology) projective and twisted connections more than is often done.  We will broadly follow the approach of Looijenga \cite[\S 1]{looijenga:2013} --- see also \cite{beilinson.bernstein:1993,tsuchimoto:1993}.

Let $E$ be a locally free sheaf on a smooth algebraic variety $X$ with tangent sheaf $T_X$.  
We can associate with this the \emph{Atiyah algebroid} $\mathcal{A}(E)$ of $E$, which is the sheaf of first order differential operators with diagonal symbol, i.e. the middle term in the top short exact sequence (known as the \emph{Atiyah sequence}) of the commutative diagram
$$\begin{tikzcd}
    0\ar[r] & \mathcal{E}nd(E)\ar[r]\ar[d, equal] & \mathcal{A}(E)\ar[r]\ar[d, hook] & T_X\ar[r]\ar[d,hook,"-\otimes \operatorname{Id}_E"] & 0 \\
    0\ar[r] & \mathcal{E}nd(E)\ar[r] & \mathcal{D}^{(1)}(E)\ar[r,"\sigma"] & T_X\otimes \mathcal{E}nd(E)\ar[r] &0,
\end{tikzcd}$$ where $\mathcal{D}^{(1)}(E)$ is the sheaf of first order differential operators on $E$, and $\sigma$ is the symbol map.
A \emph{connection} $\nabla$ on $E$ is a splitting of the Atiyah sequence
$$\begin{tikzcd} 0\ar[r] & \mathcal{E}nd(E)\ar[r] & \mathcal{A}(E)\ar[r] & T_X\ar[r] \ar[l,dashed, bend left=30, "\nabla"] & 0. \end{tikzcd}$$ 
A \emph{projective connection} on $E$ is a splitting of the push-out of the Atiyah sequence by $\mathcal{E}nd(E)\rightarrow \mathcal{E}nd(E)\big/\mathcal{O}_X$:
$$\begin{tikzcd} 0\ar[r] & \mathcal{E}nd(E)\big/\mathcal{O}_X\ar[r] & \mathcal{A}(E)\big/\mathcal{O}_X\ar[r] & T_X\ar[r] \ar[l,dashed, bend left=30, "\nabla"] & 0. \end{tikzcd}$$  
(Projective) connections are said to be \emph{flat} (or \emph{integrable}) if $\nabla$ preserves the Lie brackets on $T_X$ and $\mathcal{A}(E)$ (or $\mathcal{A}(E)/\mathcal{O}_X)$.

If $\lambda$ is a line bundle on $X$, a \emph{$\lambda$-twisted connection} on $E$ is a morphism between the Atiyah sequences of $\lambda$ and $E$, i.e. a commutative diagram:
\begin{equation}\begin{tikzcd}\label{twistedcon}
    0\ar[r] & \mathcal{E}nd(\lambda)\cong\mathcal{O}_X\ar[r]\ar[d,".w"] & \mathcal{A}(\lambda)\ar[r]\ar[d,"\nabla"] & T_X\ar[r]\ar[d,equal] & 0 \\
    0\ar[r] & \mathcal{E}nd(E)\ar[r] & \mathcal{A}(E)\ar[r] & T_X\ar[r] &0,
\end{tikzcd}\end{equation} such that $\mathcal{E}nd(\lambda)\cong\mathcal{O}_X$ is mapped to homotheties $\mathcal{O}_X\overset{.\Id_E}{\longhookrightarrow}\End(E)$.  This implies that the left vertical morphism is given by multiplication with some regular function $w$, which is referred to as the \emph{weight} of $\nabla$.

A \emph{$\lambda$-flat} connection is a $\lambda$-twisted connection such that $\nabla$ preserves the Lie brackets on $\mathcal{A}(\lambda)$ and $\mathcal{A}(E)$.
The weight $w=\nabla_1$ of a $\lambda$-flat connection is annihilated by any vector field, and hence is locally constant.  Indeed, for any local vector field $D$ that lifts to a local section $\widehat{D}$ of $\mathcal{A}(\lambda)$, we have $D(w)\Id_E=[\nabla_{\widehat{D}},\nabla_1]=\nabla_{[\widehat{D},1]}=\nabla_0=0$.

Every $\lambda$-twisted (or $\lambda$-flat) connection on $E$ canonically induces a (flat) projective connection on $E$.  
In fact, every projective connection arises in this way.  Indeed we have 
\begin{lemma} \label{projistwisted}
If the rank of $E$ is $r$, and $E$ is equipped with a projective connection $\nabla$, this lifts to a $\det(E)$-twisted connection of weight $\frac{1}{r}$.
\end{lemma}
\begin{proof}Let $\widehat{\Theta}\subset\mathcal{A}(E)$ be the pre-image of $\nabla(T_X)\subset \mathcal{A}(E)/\mathcal{O}_X$ under the projection $\mathcal{A}(E)\rightarrow\mathcal{A}(E)/\mathcal{O}_X$.  It suffices to remark that $\widehat{\Theta}$ is isomorphic to $\mathcal{A}(\det(E))$, by letting sections $\widehat{D}$ of $\widehat{\Theta}$ act on $s_1\wedge\dots\wedge s_r$ by
$\sum_{i=1}^r s_1\wedge\dots\wedge \widehat{D}(s_i)\wedge \dots \wedge s_r,$ and that here $\mathcal{O}_X\subset \widehat{\Theta}$ gets sent to $\mathcal{O}_X\subset \mathcal{A}(\det(E))$ through multiplication by $r$.
\end{proof}

Even when one is interested mainly in projective connections, the extra information of a $\lambda$-twisted connection can be helpful to control the projective ambiguity, similarly to the case of representations of central extensions of a group, compared to projective representations.  As an example, we have
\begin{lemma}\label{sum} If two vector bundles $E$ and $F$ are equipped with $\lambda$-twisted/flat connections that have the same weight $w$, then there exists a canonical  $\lambda$-twisted/flat connections with the same weight on their direct sum $E\oplus F$.
\end{lemma}
We also have
\begin{lemma}\label{descend} If $\pi:X\rightarrow Y$ is a finite \'etale morphism, $\lambda$ is a line bundle on $Y$, and $E$ is a vector bundle on $X$ equipped with a $\pi^*\lambda$-flat connection of weight $\pi^*w$, then this connection descends to a $\lambda$-flat connection on $\pi_*E$ with weight $w$.
\end{lemma}

\begin{proof}
We first consider the special case when the finite étale morphism $\pi$ is Galois. 
We denote the Galois group by $\Gamma$. In the situation of
a $\Gamma$-covering $\pi: X \rightarrow Y$, one can easily establish the following equalities
\begin{enumerate}[label=(\roman*)]
\item\label{firstprop} if $T_X$ and $T_Y$ are the tangent sheaves of $X$ and $Y$, then
$$ T_X = \pi^* T_Y \qquad \text{and} \qquad  (\pi_*T_X)^\Gamma = T_Y,$$
where $( \ \ )^\Gamma$ denotes the $\mathcal{O}_Y$-submodule  of $\Gamma$-invariant sections.
\item\label{secondprop} for any vector bundle $E$ over $X$ and for any $\gamma \in \Gamma$
$$ \gamma^* \mathcal{A}(E) = \mathcal{A}(\gamma^* E).$$
\item\label{thirdprop} for any vector bundle $F$ over $Y$
$$ (\pi_* \mathcal{A}(\pi^*F))^\Gamma = \mathcal{A}(F).$$
\end{enumerate}
Consider now a vector bundle $E$ equipped with a $\pi^* \lambda$-flat connection on $X$. Then, considering that 
$\pi^* \lambda$ is $\Gamma$-invariant, we obtain that for every $\gamma \in \Gamma$ the vector bundle
$\gamma^* E$ is equipped with a $\pi^* \lambda$-flat connection  --- here we have used properties \ref{firstprop} and \ref{secondprop}.
Since the weight $\pi^*w$ does not depend on $\gamma$, we can apply Lemma \ref{sum} and we obtain a $\pi^* \lambda$-flat
connection on the direct sum 
$$ \bigoplus_{\gamma \in \Gamma} \gamma^* E = \pi^* ( \pi_* E),$$
i.e. a Lie algebra homomorphism
$$\begin{tikzcd} \mathcal{A}(\pi^* \lambda) \ar[r]& \mathcal{A}(\pi^*(\pi_* E)).\end{tikzcd}$$
Now we apply $(\pi_* \  \ )^\Gamma$ as well as property \ref{thirdprop} and we obtain a $\lambda$-flat connection on $\pi_* E$.

Next, we consider a general étale cover $\pi: X \rightarrow Y$. Then there exists a Galois cover $\varphi: Z \rightarrow X$
such that $\pi' = \pi \circ \varphi : Z \rightarrow X \rightarrow Y$ is also Galois --- actually $\pi'$ can be constructed as Galois closure of $\pi$. We can now apply the previous result to the vector bundle $\varphi^*E$ over the Galois cover 
$Z$ and we obtain a 
$\lambda$-flat connection on $\pi'_*( \varphi^* E) = \pi_*( E \otimes \varphi_* \mathcal{O}_Z)$. But 
$\varphi_* \mathcal{O}_Z$ contains $\mathcal{O}_X$ as a direct summand, hence we can decompose 
$E \otimes \varphi_* \mathcal{O}_Z = E \oplus \mathcal{F}$, where $\mathcal{F}$ is a residual vector bundle. Now we 
project the $\lambda$-flat connection onto the factor $\pi_*E$ and we obtain the desired result.
\end{proof}

It is sometimes useful to calibrate the weight of a twisted connection.  We have
\begin{lemma}\label{scale} For any non-zero integer $n$, there exists a canonical isomorphism 
of Atiyah sequences
$$\begin{tikzcd}
    0\ar[r] & \mathcal{O}_X\ar[r]\ar[d,".n"] & \mathcal{A}(\lambda)\ar[r]\ar[d,"\cong"] & T_X\ar[r]\ar[d,equal] & 0 \\
    0\ar[r] & \mathcal{O}_X\ar[r] & \mathcal{A}(\lambda^{\otimes n})\ar[r] & T_X\ar[r] &0,
\end{tikzcd}$$
where $\widehat{D}\in \mathcal{A}(\lambda)$ acts on $s_1\otimes\dots\otimes s_n\in \mathcal{O}(\lambda^{\otimes n})$ by
$\sum_{i=1}^n s_1\otimes \dots\otimes \widehat{D}(s_i)\otimes\dots\otimes s_n.$
\end{lemma}
As a consequence, for a vector bundle $E$, a $\lambda$-twisted connection of weight $w$ is equivalent to a $\lambda^{\otimes n}$-twisted connection of weight $\frac{1}{n}w$.  This implies that when considering $\lambda$-twisted connections, one can ignore any torsion factors in the line bundle $\lambda$.

Inspired by this, one can construct the formal Atiyah algebroid $\mathcal{A}(\lambda^{\otimes s})$, for any scalar $s\in \mathbb{C}$.  By taking $s=w$, one obtains (if $s\neq 0$) a diagram like (\ref{twistedcon}), where now the weight is $1$.  Often this is taken to be part of the definition, see e.g. \cite[\S 2]{tsuchimoto:1993}, where a $\lambda$-flat connection with weight 1 is referred to as an action of the Atiyah algebroid on the vector bundle.

If $E$ has a $\lambda$-twisted (or $\lambda$-flat) connection of weight $0$, the map $\nabla$ will factor through $\mathcal{A}(\lambda)/\mathcal{O}_X\cong T_X$, hence gives rise to an ordinary (flat) connection.  Also if $\lambda\cong \mathcal{O}$, we have that canonically $\mathcal{A}(\lambda)\cong \mathcal{O}\oplus T_X$, and hence $\nabla|_{T_X}$ is an ordinary (flat) connection.  Remark that this last fact, combined with Lemma \ref{scale}, imply that if a rank $r$ vector bundle $E$ carries a $\lambda$-flat connection with weight $1$, then $\frac{c_1(E)}{r}=c_1(\lambda)\in H^2(X,\mathbb{Q})$, see \cite[Lemma 5]{MOP:2013}.

In particular, one can remark that the pull-back of any line bundle $\lambda$ to $|\lambda^{\times}|$, the total space of $\lambda$ minus the zero section, is canonically trivial.  Since ($\lambda$-twisted) connections also pull back, this associates with a $\lambda$-twisted/flat connection an ordinary (flat) connection on $|\lambda^{\times}|$.

\begin{remark}The entire discussion above was set in an algebro-geometric context, but a parallel theory exists in differential geometry.  In that setting, a flat (projective) connection gives rise to a (projective) representation of the fundamental group of $X$, called the \emph{monodromy representation}.  A $\lambda$-flat connection will give rise to a representation of a particular central extension of the fundamental group of $X$, which can itself be interpreted as the fundamental group of $|\lambda^{\times}|$.  We will work in the complex setting throughout the paper, and freely associate this monodromy representation to the algebraic connections we are working with through analytification.
\end{remark}

\subsection{Flat morphisms}A morphism $\Phi$ between bundles $E$ and $F$ over $S$, equipped with connections $\nabla^E$ and $\nabla^F$, is said to be \emph{flat} (or to preserve the connections), if for every open $U
\subset S$, every $X\in T_S(U)$, and every $s\in E(U)$, we have $$\nabla^F_X(\Phi s)=\Phi\left(\nabla^E_X s\right).$$
This is equivalent to $\Phi$, thought of as a section of $F\otimes E^*$, being flat for the tensor connection, i.e. $\nabla^{F\otimes E^*}_X\Phi=0$ for all vector fields $X$.

If $E$ and $F$ are equipped with projective connections, then $\Phi$ is flat if locally $$\widetilde{\nabla}^F_X(\Phi s)-\Phi\left(\widetilde{\nabla}^E_X s\right)=\omega(X) \Phi(s),$$
for some one-form $\omega$, where $\widetilde{\nabla}^F$ and $\widetilde{\nabla}^E$ are local connections lifting the projective connections on $E$ and $F$.
Flat morphisms between bundles with projective connections have constant rank.  Note that one could develop the notion of flat morphism also for $\lambda$-twisted connection, but the projective case will suffice for our purposes.
\subsection{(Twisted) \texorpdfstring{$\mathcal{D}$}{D}-modules}
For completeness, we mention the more common guises that the structure of a $\lambda$-flat connection takes in the literature.  It is a standard result that a flat connection on a bundle $E$ is equivalent to a (left) $\mathcal{D}$-module structure on $E$, where $\mathcal{D}$ is the sheaf of differential operators.

We can similarly look at $\mathcal{D}_{\lambda}$, which is the sheaf of differential operators on a line bundle $\lambda$.  This is an example of a sheaf of twisted differential operators, or \emph{tdo} for short.  There is now again an equivalence between left $\mathcal{D}_{\lambda}$-module structures on coherent $\mathcal{O}_X$-modules $E$ (which necessarily have to be locally free), and $\lambda$-flat connections on $E$ with weight $1$.

\subsection{Flat projective connections with finite monodromy}
Let $E$ be a vector bundle of rank $r$ equipped with a flat projective connection $\nabla$. Then, fixing a point $x \in X$, we can consider 
the associated projective monodromy representation of the topological fundamental group
$$\begin{tikzcd}   \rho_{(E, \nabla)} : \pi_1(X,x) \ar[r]& \mathbb{P}\mathrm{GL}(r).\end{tikzcd}$$
We will in particular be interested in flat projective connections that have finite mo\-no\-dro\-my, i.e. $\mathrm{im} \  \rho_{(E, \nabla)}$ is a finite
subgroup of $\mathbb{P}\mathrm{GL}(r)$.

Let $\nabla$ be a flat projective connection on
a rank $r$ vector bundle $E$. 
We say that $\nabla$ is \emph{trivial} if its monodromy representation is trivial.

With the above notation we have the following two lemmas, which can easily be deduced from \cite[Theorem 7.6]{ramanan:2005}.

\begin{lemma} \label{trivialmonodromy}
The flat projective connection $\nabla$ has trivial monodromy 
if and only if there exists a line bundle $L$ such that $E\cong L^{\oplus r}$, with the projective connection $\nabla$ induced by the $L$-flat connection of weight $1$

$$\begin{tikzcd}
    0\ar[r] & \mathcal{O}_X\ar[r]\ar[d,".\mathrm{Id}"] & \mathcal{A}(L)\ar[r]\ar[d,"\Delta"] & T_X\ar[r]\ar[d,equal] & 0 \\
    0\ar[r] & \mathcal{E}nd(L^{\oplus r})\ar[r] & \mathcal{A}(L^{\oplus r}) \ar[r] & T_X\ar[r] &0,
\end{tikzcd}$$ where $\Delta$ is the diagonal embedding. In that case we will say that the projective connection $\nabla$ is trivial.
\end{lemma}

\begin{lemma} \label{finitebecomestrivial}
The flat projective connection $\nabla$ has finite monodromy if and only if there 
exists a finite \'etale morphism $f:\widetilde{X}\rightarrow X$ such that the flat projective connection $f^*\nabla$ on $f^*E$ is trivial.
\end{lemma}

\section{Determinants and slopes}\label{dets}
\subsection{} In order to combine various projective representations, we need to be able to compare the slopes of the vector bundles they are defined on.  

Let $\mathcal{R}_g$ be the moduli stack of connected \'etale double covers $\calCt$, with involution $\sigma$, of smooth curves $\calC$ of genus $g$.  There are natural forgetful morphisms $\mathcal{R}_g\rightarrow \mathcal{M}_{2g-1}$, $\pi_g:\mathcal{R}_g\rightarrow \mathcal{M}_g$ (given by forgetting the base curve or the covering curve respectively) and a morphism $\mathcal{R}_g\rightarrow \mathcal{A}_{g-1}$ (given by taking the Prym variety).  With each of these one can associate a Hodge line bundle, and the ones coming from $\mathcal{M}_g$ and $\mathcal{A}_{g-1}$ are isomorphic (see e.g. \cite[Lemma 7.18]{casalaina-laza:2009}) --- we shall refer to this line bundle as $\lambda$ (the line bundle coming from $\mathcal{M}_{2g-1}$ is twice $\lambda$).

Let $\pi_J:J^*\rightarrow \mathcal{M}_g$ be the universal Jacobian torsor parametrizing degree $g-1$ line bundles over smooth genus $g$ curves $\calC$, and let $\Theta \subset J^*$ be the relative canonical theta divisor.  Likewise, let $\pi_{P}:P^{\rm even}\rightarrow \mathcal{R}_g$ be the Prym torsor, parametrizing line bundles on $\calCt$ that satisfy $\sigma^*L\cong L^*\otimes K_{\calCt}$, such that $H^0(\calCt,L)$ is even-dimensional. 
On $P^{\rm even}$ we have a canonical divisor $\Xi$ and for the double cover $\widetilde{\mathcal{C}}\to\mathcal{C} $ we will denote by $\tilde{g}=2g-1$ the genus of $\widetilde{\mathcal{C}}$.
We are interested in $\pi_{J*}\Theta^k$ and $\pi_{P*}\Xi^k$, for $k>0$.  
\begin{proposition}\label{slope} We have over $\mathcal{R}_g$ 
$$\pi_g^*\det \left(\pi_{J*}\Theta^k\right)\cong \lambda^{\otimes\frac{1}{2}k^g(k-1)}\hspace{.5cm}
\text{and} \hspace{.5cm}\operatorname{rk} \left(\pi_{J*} \Theta^k\right)=k^g,$$
and 
$$\det \left(\pi_{P*}\Xi^k\right)\cong 
\lambda^{\otimes\frac{1}{2}k^{g-1}(k-1)}\hspace{.5cm}
\text{and} \hspace{.5cm}\operatorname{rk} \left(\pi_{P*}\Xi^k\right)=k^{g-1}.$$
\end{proposition}
\begin{proof}The formulas for the ranks are just instances of a classically known formula for principally polarized abelian varieties, see e.g. \cite[Corollary 3.2.8]{birkenha-lange:2004}; this was first obtained in terms of complex analysis, and can be shown algebraically through an application of the Hirzebruch-Riemann-Roch formula 
and the Kodaira vanishing theorem.  The formula for the determinant of $\pi_{J*}\Theta^k$ is given in \cite[Theorem C]{kouvidakis:2000}.   The formula for the determinant of $\pi_{P*}\Xi^k$ can be obtained similarly, based on a computation done by Moret-Bailly in \cite[page 256]{moret-bailly:1985} and developed also in the alternative proof of Theorem C in \cite[Section 2.4]{kouvidakis:2000}.

Let $\pi_{P_0}:\prym\to \mathcal{R}_g$ denote the relative degree 0 Prym variety over $\mathcal{R}_g$, that is the connected component of the kernel of the norm map $\operatorname{Nm}$ that contains the origin. Recall also that we set $\pi_P: P^{\rm even}\to \mathcal{R}_g$ for the universal degree $\tilde{g}-1$ Prym torsor defined above. Let $\widetilde{\mathcal{R}}_g$ denote the moduli space parametrizing double covers $\widetilde{C}\to C$ in $\mathcal{R}_g$ together with a theta characteristic $\kappa$ on $\widetilde{C}$ such that $\operatorname{Nm}(\kappa)=\omega_C$ with $h^0(\widetilde{C},\kappa)$ even. These are exactly the theta characteristics that send, via tensor product, the degree zero Prym variety $\prym$ to $P^{\rm even}$. We will call $\alpha:\widetilde{\mathcal{R}}_g \to \mathcal{R}_g$ the natural forgetful map. We will also denote by $\widetilde{\prym}$ and $\widetilde{P}^{\rm even}$ the pullback of the universal families to $\widetilde{\mathcal{R}}_g$, and by $\phi$ the map that sends $L\in \widetilde{\prym}$ sitting over the moduli point $[\widetilde{C}\to C, \kappa]$ to $L\otimes \kappa\in \widetilde{P}^{\rm even}$. 
The situation is resumed in the diagram (\ref{diagrammichele}), which is the analogue in terms of Prym varieties of the diagram in terms of Jacobians and Picard varieties that appears in \cite[Section 4.1]{kouvidakis:2000}. The maps $\widetilde{\pi}_{P_0}$ and $\widetilde{\pi}_P$ are just the pull-backs of the Prym fibrations to $\widetilde{\mathcal{R}}_g$, and we will denote respectively by $\tilde{s}:\widetilde{\mathcal{R}}_g \to \widetilde{\prym}$ and $\tilde{\sigma}:\widetilde{\mathcal{R}}_g \to \widetilde P^{\rm even} $ the zero section and the section that sends $[\widetilde{C}\to C, \eta]$ to $\eta$:

\begin{equation}\label{diagrammichele}
\begin{tikzcd}
\prym \ar[d,swap,"\pi_{P_0}"] & {\widetilde{\prym}} \ar[l,swap,"\gamma"] \ar[rr,"\phi"] \ar[dr,swap, "\widetilde \pi_{P_0}"] & & \widetilde P^{\rm even} \ar[r,"\delta"] \ar[dl,"\widetilde \pi_P"]
    & P^{\rm even} \ar[d,"\pi_P"] \\
\mathcal{R}_g & &\widetilde{\mathcal{R}}_g \ar[ll,"\alpha"] \ar[rr,swap,"\alpha"]
   & & \mathcal{R}_g.\\
\end{tikzcd}
\end{equation}
By \cite[Lemma 3.2]{kouvidakis:2000}, we have that $\Omega_{\widetilde{\prym}/\widetilde{\mathcal{R}}_g}\cong \tilde{\pi}_{P_0}^*\tilde{s}^*\Omega_{\widetilde{\prym}/\widetilde{\mathcal{R}}_g}$. The map $\phi$ is an isomorphism, hence we also get

\begin{equation}\label{kouvipullback}
\Omega_{\widetilde{P}^{\rm even}/\widetilde{\mathcal{R}}_g}\cong \tilde{\pi}_{P}^*\tilde{\sigma}^*\Omega_{\widetilde{P}^{\rm even}/\widetilde{\mathcal{R}}_g}.
\end{equation}

Thus, we observe that $\Omega^\vee_{\widetilde{P}^{\rm even}/\widetilde{\mathcal{R}}_g} \cong \widetilde \pi_{P}^*E$, where $E$ is a vector bundle on $\widetilde{\mathcal{R}_g}$.
By Equation \ref{kouvipullback}, $E$ can be seen as $\tilde{\sigma}^*\Omega^\vee_{\widetilde{P}^{\rm even}/\widetilde{\mathcal{R}}_g}$. On the other hand, since $\Omega_{\widetilde{P}^{\rm even}/\widetilde{\mathcal{R}}_g}$ is trivial on the (complete, connected and reduced) fibers of $\widetilde \pi_P$, $E$ is also isomorphic to $\widetilde \pi_{P *}\Omega^\vee_{\widetilde{P}^{\rm even}/\widetilde{\mathcal{R}}_g} $, and there is an isomorphism $\widetilde \pi_{P}^*(\widetilde \pi_{P *}\Omega^\vee_{\widetilde{P}^{\rm even}/\widetilde{\mathcal{R}}_g}) \cong \Omega^\vee_{\widetilde{P}^{\rm even}/\widetilde{\mathcal{R}}_g} $.

Let us set $\widetilde{\Xi}:= \delta^*\Xi$.  We will apply the Grothendieck-Riemann-Roch theorem to the fibration $\widetilde \pi_P: \widetilde{P}^{\rm even}\to \widetilde{\mathcal{R}}_g$ and obtain

\newcommand{\td}{\mathrm{Td}}
\newcommand{\ch}{\mathrm{ch}}

\begin{equation} 
\ch(\widetilde \pi_{P !}(\widetilde{\Xi}^k))= \widetilde \pi_{P *}(\ch(\widetilde\Xi^k) \cdot \td(\Omega^\vee_{\widetilde{P}^{\rm even}/\widetilde{\mathcal{R}}_g})). 
\end{equation}

On the left side, the higher direct images vanish so we are left with $c_1(\widetilde \pi_{P *}(\widetilde\Xi^k))$. On the right side: for the Todd class, we use the projection formula on Equation \ref{kouvipullback}, and hence we get 

\begin{equation}\label{grr}
c_1(\widetilde \pi_{P *}(\widetilde\Xi^k))=\widetilde \pi_{P *}(\ch(\widetilde\Xi^k))\cdot \td (\tilde{\sigma}^*\Omega^\vee_{\widetilde{P}^{\rm even}/\widetilde{\mathcal{R}}_g}).
\end{equation}

Now we have $\td (\tilde{\sigma}^*\Omega^\vee_{\widetilde{P}^{\rm even}/\widetilde{\mathcal{R}}_g}) = 1 - \frac{c_1(\tilde{\lambda})}{2}$ plus higher degree terms, where $\tilde{\lambda}$ is the pull-back of $\lambda$ to $\widetilde{\mathcal{R}_g}$ (recall that the Hodge line bundles coming from $\mathcal{M}_g$ and $\mathcal{A}_{g-1}$ are the same class $\lambda$ on $\mathcal{R}_g$). Hence, by using the Poincaré formula and the fact that the relative dimension of the fibration is $g-1$, we obtain that the degree one part of equality (\ref{grr}) reads

\begin{equation}\label{grrformula}
c_1(\widetilde \pi_{P *}(\widetilde\Xi^k))=\frac{k^g}{g!} \widetilde \pi_{P *}c_1^{g}(\widetilde\Xi)-\frac{k^{g-1}}{2}c_1(\tilde{\lambda}).
\end{equation}

We will now use the formula in Equation (\ref{grrformula}) in order to compute the class of $\det(\pi_{P*} \Xi^k)$.
 Since the Picard group of $\mathcal{R}_g$ is generated by $\lambda$, see \cite{FarkasLudwig}, following the argument of \cite[4.2]{kouvidakis:2000}, we can suppose that $c_1(\pi_{P *}\Xi^k)=s(k)c_1(\lambda)$ and that $\pi_{P *}c_1^{g}(\Xi)=tc_1(\lambda)$, for some coefficients $s(k), t\in \mathbb{Z}$. This implies in turn that $c_1(\widetilde \pi_{P *}\widetilde\Xi^k)=s(k) c_1(\tilde\lambda)$ and that $\widetilde \pi_{P *}c_1^{g}(\widetilde\Xi)=t c_1(\tilde\lambda)$, on $\widetilde{\mathcal{R}}_g.$ Hence $s(k)=\frac{k^g}{g!}t-\frac{k^{g-1}}{2}$. The line bundle $\widetilde \pi_{P *} \widetilde{\Xi}$ has by definition a nowhere vanishing section, hence $s(1)=0$ and $t=\frac{g!}{2}$. This implies that
$\det \left(\pi_{P *}(\Xi^k)\right)= 
\frac{1}{2}k^{g-1}(k-1) \lambda.$ In multiplicative notation, this is our claim.
\end{proof}

\section{Flatness of the relative Oxbury-Pauly isomorphism}\label{flatness}
\subsection{Conformal embeddings}Given a family of  smooth projective curves $\calC\rightarrow S$, a simple, simply-connected reductive algebraic group $G$, and a level $k\in\mathbb{N}$, one can consider the bundles of non-abelian theta functions of level $k$, which are obtained as $\pi_*\LL_G^k$, where $\pi:\MM_{G,S}\rightarrow S$ is the moduli space of stable $G$-torsors over $\calC$, and $\LL_G$ is the relatively ample generator of the moduli space of $G$-bundles $Pic(\MM_G/S)$.  This bundle $\pi_*\LL_G^k$ can be equipped with the flat projective Hitchin connection, see e.g. \cite{hitchin:1990,BBMP:2020,BMW:2023}. 

Projectively these bundles can also be identified with the bundles of conformal blocks arising in the Wess-Zumino-Witten model.  On the latter there is also a natural connection, or to be precise a $\lambda$-flat connection, for $\lambda$ the Hodge determinant bundle, with weight $\frac{c}{2}$, where $c=\frac{k\dim G}{k+h^{\vee}_G}$ is the central charge from the WZW model (here $h^{\vee}_G$ is the dual Coxeter number of $G$, which is $n$ for $G=\SL_n$) \cite{TUY:1989,looijenga:2013}.   Projectively these connections are isomorphic \cite{laszlo:1998}.

If $\phi: G\rightarrow H$ is a morphism of such groups (it suffices to give the morphism $\mathfrak{g}\rightarrow \mathfrak{h}$ between their simple Lie algebras), we can associate with it an integer $d_{\phi}$, known as the \emph{Dynkin index}.  It can be expressed in a number of equivalent ways --- a useful one for our purposes is as follows: from $\phi$ we get, by extension of structure group, a natural morphism $\widetilde{\phi}:\MM_{G}\rightarrow\MM_{H}$, and we have that $\widetilde{\phi}^*(\LL_H)\cong \LL_G^{d_{\phi}}$.  

The morphism $\phi$ is said to be a \emph{conformal embedding} if the central charge for $G$ at level $d_{\phi}$ equals the central charge for $H$ at level $1$ --- this numerical condition ensures the compatibility of the Segal-Sugawara constructions for the affine Lie algebras $\widehat{\mathfrak{g}}$ and $\widehat{\mathfrak{h}}$. (Note that it only may be an embedding for the Lie algebras.)  It was shown by Belkale that, given a morphism $\phi$ as above, the natural map between bundles of conformal blocks (at level $d_{\phi}$ for $G$, and level $1$ for $H$) is flat if $\phi$ is a conformal embedding \cite[Proposition 5.8]{belkale:2009}.

Moreover, if a family of curves $\calCt\rightarrow S$ is equipped with a Galois action of a group $\Gamma$ (relative to $S$), and $\Gamma$ also acts on $G$, one can consider the associated bundles of \emph{twisted conformal blocks}, which also come equipped with a $\lambda$-flat connection (see \cite{szcesny:2006,damiolini:2020,deshpande.mukhopadhyay:2023,hong.kumar:2018}).  It was shown in \cite[Theorem A.4.1]{BBMP:2024} that $\Gamma$-equivariant conformal embeddings (for $\Gamma$ cyclic) still give rise to flat morphisms between the corresponding bundles of twisted conformal blocks. Furthermore, in the case of $G=\SL_n$, and $\Gamma=\mathbb{Z}/2\mathbb{Z}$ acting without fixed points on $\calCt$, a flat projective connection (dubbed the \emph{Prym-Hitchin connection}) was constructed, through a heat operator, on non-abelian theta functions on the moduli space of anti-invariant bundles over $\calCt$ in \cite{BBMP:2024}, for which there is again a Laszlo correspondence with the connection on the relevant bundles of twisted conformal blocks.

A list of conformal embeddings was given in \cite{schellekens-warner:1986}.  The one of relevance for us is one of the simplest ones, given on the level of Lie algebras by $\mathfrak{sl}_2\hookrightarrow \mathfrak{sl}_3$, induced by the adjoint representation $\ad:\mathfrak{sl}_2\rightarrow \End(\mathfrak{sl}_2)$, which has Dynkin index $4$.

It is now an elementary --- but for our purposes crucial --- observation that this embedding $\mathfrak{sl}_2\hookrightarrow \mathfrak{sl}_3$ is equivariant in two different ways: besides for the trivial group, we can also consider the action of $\Gamma=\mathbb{Z}/2\mathbb{Z}$, which acts trivially on $\mathfrak{sl}_2$, but via a non-trivial involution on $\mathfrak{sl}_3$.  To this end, we identify $\mathbb{C}^3$ with $\mathfrak{sl}_2$ (traceless $2\times 2$ matrices), and use the invariant symmetric bilinear form $\langle A,B\rangle=\tr(AB)$.  If we identify $\mathbb{C}^3\cong\mathfrak{sl}_2$ by
$$\begin{tikzcd}\begin{pmatrix}a\\ b\\ c\end{pmatrix}\ar[r,mapsto] & \begin{pmatrix}a & b \\ c & -a\end{pmatrix} \end{tikzcd},$$
this bilinear form is represented by $$
\left\langle
\begin{pmatrix}a \\ b\\ c\end{pmatrix},
\begin{pmatrix}\widetilde{a}\\ \widetilde{b}\\ \widetilde{c}\end{pmatrix}
\right\rangle=
\begin{pmatrix}a & b & c\end{pmatrix}
\begin{pmatrix} 2 & 0 & 0 \\ 0 & 0 & 1\\ 0& 1 & 0\end{pmatrix}
\begin{pmatrix}\widetilde{a}\\ \widetilde{b}\\ \widetilde{c}\end{pmatrix}.$$

The involution on $\mathfrak{sl}_3$ is now simply given by \begin{equation}\begin{tikzcd}\label{inv-sl3}X\in \mathfrak{sl}_3\ar[r, maps to]& -X^{\tau},\end{tikzcd}\end{equation} where $X^{\tau}$ is the adjoint with respect to $\langle .\, ,.\rangle$, i.e. $\langle A,XB\rangle=\langle X^{\tau}A,B\rangle$, or, equivalently,
$$X^{\tau}=\begin{pmatrix} 2 & 0 & 0 \\ 0 & 0 & 1\\ 0& 1 & 0\end{pmatrix}^{-1}X^t\begin{pmatrix} 2 & 0 & 0 \\ 0 & 0 & 1\\ 0& 1 & 0\end{pmatrix},$$ where $X^t$ is the standard matrix transpose of $X$.  We now have
\begin{lemma}The image under $\ad$ of $\mathfrak{sl}_2$ in $\mathfrak{sl}_3$ is invariant under the involution (\ref{inv-sl3}) on $\mathfrak{sl}_3$.
\end{lemma}
\begin{proof}It suffices to observe that, for $A,B,C \in\mathfrak{sl}_2$, we have
$$-\langle \ad(C)A,B\rangle=-\tr([C,A]B)=-\tr(CAB)+\tr(ACB)$$ and
$$\langle A,\ad(C)B\rangle=\tr(A[C,B])=\tr(ACB)-\tr(ABC)=\tr(ACB)-\tr(CAB).$$ Therefore we have
$\langle -\ad(C)A,B\rangle=\langle A, \ad(C)B\rangle,$ hence $\ad(C)=-\ad(C)^{\tau}$.
\end{proof}
We consider now again the morphism $\pi_g:\mathcal{R}_g\rightarrow \mathcal{M}_g$ from Section \ref{dets}. We have on $\mathcal{M}_g$ the 
two vector bundles $\mathbb{V}(\SL_2, 4)$ and $\mathbb{V}(\SL_3, 1)$ of non-abelian theta functions at level four 
(resp. one) on the moduli space of $\SL_2$ (resp. $\SL_3$)-bundles. On $\mathcal{R}_g$ we have the vector bundle 
$\mathbb{V}^{\rm tw}(\SL_3, 1)$ whose fiber over a double étale cover $p: \calCt \rightarrow \calC$ equals the space
$H^0(\mathcal{N}^{+,{\rm ss}}_{\SL_3}, \mathcal{P}_3)$
of non-abelian theta functions at level one on the moduli space $\mathcal{N}^{+,ss}_{\SL_3}$ of (symmetric) anti-invariant $\SL_3$-bundles over $\calCt$ (see \cite{BBMP:2024} for the precise definitions of $\mathcal{N}^{+,{\rm ss}}_{\SL_3}$ and the Pfaffian line bundle $\mathcal{P}_3$).
These are equipped with the Hitchin \cite{BBMP:2020} and Prym-Hitchin \cite[Theorem 5.1.2]{BBMP:2024} flat projective connections.  
The discussion above shows

\begin{proposition}\label{conf-emb-flat}
The two natural maps $E \mapsto \mathrm{End}_0(E)$  and $E \mapsto p^* \mathrm{End}_0(E)$ from $\mathcal{M}_{\SL_2}$ to 
$\mathcal{M}_{\SL_3}$ (resp. $\mathcal{N}^{+,{\rm ss}}_{\SL_3}$) induce the following morphisms of vector bundles 
$$\begin{tikzcd} \mathbb{V}(\SL_3, 1)  \ar[r]& \mathbb{V}(\SL_2, 4) \end{tikzcd} $$ on $\mathcal{M}_g$ and $$\begin{tikzcd} \mathbb{V}^{\rm tw}(\SL_3, 1) \ar[r]&\pi_g^*  \mathbb{V}(\SL_2, 4) \end{tikzcd}$$ on $\mathcal{R}_g$. These two morphisms are projectively flat for the Hitchin (resp. Prym-Hitchin)
connections.
\end{proposition}

\begin{proof}The flatness of the first morphism is a direct application of Laszlo's comparison theorem \cite{laszlo:1998} and \cite[Proposition 5.8]{belkale:2009}.  For the flatness of the second morphism, it suffices to remark that under the identification given by \cite[Theorem 12.1]{hong.kumar:2018} the two spaces of non-abelian
theta functions correspond to the twisted conformal blocks associated to $\mathfrak{sl}_3$ with the involution \eqref{inv-sl3} and to 
$\mathfrak{sl}_2$ with the trivial involution and the corresponding morphism between conformal blocks is induced by the embedding
$\ad : \mathfrak{sl}_2 \hookrightarrow \mathfrak{sl}_3$.
Moreover, the bundle of twisted conformal blocks for $\mathfrak{sl}_2$ with trivial twist on $\mathcal{R}_g$ is exactly the pull-back of the bundle of ordinary conformal blocks for $\mathfrak{sl}_2$ on $\mathcal{M}_g$. We obtain projective flatness of the morphism by combining \cite[Theorem A.4.1]{BBMP:2024}, applied to the $\mathbb{Z}/2 \mathbb{Z}$-equivariant conformal embedding
$\ad : \mathfrak{sl}_2 \hookrightarrow \mathfrak{sl}_3$, with the twisted analogue of Laszlo's comparison 
theorem \cite[Theorem 6.2.3]{BBMP:2024}.
\end{proof}
 
\subsection{(Anti-invariant) level-rank duality}
The second main ingredient to establish the flatness of the abelianization morphism is level-rank duality (also known as strange duality).  Originating in the physics of the WZW model \cite{naculich-schnitzer:1990a,naculich-schnitzer:1990b}, the mathematical formulation of level-rank duality \cite{donagi-tu:1994, beauville:1995} observes that, given a smooth curve $\calC$, there is a natural map, via the tensor product of bundles
$$\begin{tikzcd}H^0(\mathcal{M}_{\GL_{kr}}^*,\Theta) \ar[r]& H^0(\mathcal{M}_{\SL_r},\LL_{\SL_r}^k)\otimes H^0(\mathcal{M}_{\GL_k}^*,\Theta^r)\end{tikzcd}$$ (here $\mathcal{M}_{\GL_k}^*$ is the moduli space of rank-$k$ bundles with degree $k(g-1)$, which comes equipped with a canonical theta divisor $\Theta$).  Since $h^0(\mathcal{M}_{\GL_{kr}}^*,\Theta)=1$, this gives a natural map, canonically defined up to scalars
$$\begin{tikzcd}H^0(\mathcal{M}_{\GL_k}^*,\Theta^r)^\vee \ar[r]& H^0(\mathcal{M}_{\SL_r},\LL_{\SL_r}^k).\end{tikzcd}$$  
The level-rank duality conjecture, which asserts that this is an isomorphism, was proven by Marian and Oprea \cite{marian.oprea:2007} and Belkale \cite{belkale:2008,belkale:2009}.  Belkale in fact showed that, when considering this morphism for families $\calC\rightarrow S$ of smooth curves, the resulting morphism of bundles preserves the flat projective connections on both sides.  In the case of $k=1$ (the only case we will need), this in fact was already implicitly shown in \cite{beauville.narasimhan.ramanan:1989} in terms of the action of the theta group.  

An anti-invariant version of level-rank duality, for unramified double covers of smooth curves $\calCt\rightarrow \calC$ with involution $\sigma$, was recently proposed by the authors in \cite{BBMP:2024}.  The role of $\mathcal{M}_{\SL_r}$ is now played by the moduli space 
$\mathcal{N}^{+,{\rm ss}}_{\SL_r}$ of bundles $E$ on $\calCt$ with trivial determinant and further equipped with an isomorphism 
$\sigma^*E\cong E^\vee$.  The analogue of $\mathcal{M}_{\GL_r}^*$ is the moduli space of bundles $E$ on $\calCt$ with an isomorphism $\sigma^*E\cong E^\vee \otimes K_{\calCt}$ --- for $r=1$, this is a torsor over the corresponding Prym variety\footnote{The statements in \cite{oxbury.pauly:1996} do not involve this torsor, but just work with the Prym varieties themselves.  This difference is inconsequential for a single curve, but crucial when considering families.} for the cover $\calCt\rightarrow \calC$.
It was shown to be an isomorphism for $k=1$, and to be flat for all levels (note that the ramified case, for $k=1$, was also recently established by Zelaci in \cite{zelaci:2025}).  In particular we have \cite[Corollary 7.3.2]{BBMP:2024} that on $\mathcal{R}_g$, the morphism 
$$\begin{tikzcd}(\pi_{P*} \Xi^r)^\vee  \ar[r]& \mathbb{V}^{\rm tw}(\SL_r,1)\end{tikzcd}$$
is a projectively flat isomorphism for any rank $r$.

Combining these facts  with Proposition \ref{conf-emb-flat} 
leads to
\begin{proposition}\label{compositemorph}
The two composite morphisms
\begin{equation}
\label{flatcombined}
\begin{tikzcd}
\left(\pi_{J*}\Theta^3\right)^{\vee}\ar[r] &  \mathbb{V}(\SL_3, 1) \ar[r] & \mathbb{V}(\SL_2, 4) \\
\left(\pi_{P*}\Xi^3\right)^{\vee} \ar[r] & \mathbb{V}^{\rm tw}(\SL_3, 1)  \ar[r] & \pi_g^* \mathbb{V}(\SL_2, 4) 
\end{tikzcd}
\end{equation}
are projectively flat over $\mathcal{M}_g$ and $\mathcal{R}_g$ respectively. 
\end{proposition}

\section{Abelianization and finiteness of the monodromy}\label{finiteness}
\subsection{The Mumford-Welters connection on the theta bundle} \label{mumfordweltersconnection}
In this subsection we recall the main results on theta bundles 
associated to a family $\pi : \mathcal{A} \rightarrow S$ of abelian varieties.
Given a relatively ample line bundle $\mathcal{L}$ over $\mathcal{A}$ we recall that there exists a projective flat 
connection $\nabla$, called the Mumford-Welters connection, on the theta bundle $\pi_* \mathcal{L}$ over $S$. The standard references are \cite[\S 6]{mumford:1966} and \cite[Proposition 2.7]{welters:1983}.

In this paper we only consider line bundles of the form $\mathcal{L} =\Theta^n$, where $\Theta$ is a principal polarization and $n \in \mathbb{N}^*$ (actually we only need the case $n=3$). We consider the Heisenberg group $\mathcal{G}(n)$, defined as a certain non-abelian central extension
$$ \begin{tikzcd}0 \ar[r] &\mathbb{G}_m \ar[r]& \mathcal{G}(n) \ar[r]& H(n) \ar[r]& 0,\end{tikzcd} $$
of the abelian group  $H(n) = \left( \mathbb{Z} / n \mathbb{Z} \right)^{g} \times \left( \mathbb{Z} / n \mathbb{Z} \right)^{g}$ equipped with the standard symplectic form. An essential notion underlying the Mumford-Welters connection is the notion of \emph{theta structure}, 
i.e. an isomorphism inducing the identity on the central $\mathbb{G}_m$ between the two groups 
$$\begin{tikzcd} \mathcal{G}(\mathcal{L}_s) \ar[r,"\sim"] & \mathcal{G}(n),\end{tikzcd}$$
where $\mathcal{G}(\mathcal{L}_s)$ is the Mumford group associated to  $(\mathcal{A}_s, \mathcal{L}_s)$. We recall that
fixing a theta structure for $(\mathcal{A}_s, \mathcal{L}_s)$ allows to canonically (up to homotheties) identify the 
space of global sections
$$\begin{tikzcd}H^0(\mathcal{A}_s, \mathcal{L}_s)\ar[r,"\sim"] & V\end{tikzcd}$$
with the Schrödinger representation $V$ of $\mathcal{G}(n)$, the unique irreducible representation 
of $\mathcal{G}(n)$ of level $1$. In terms of families of abelian varieties, this implies that on the 
étale cover $p : \widetilde{S} \rightarrow S$ parameterizing
abelian varieties of the family $S$ together with a theta structure, the projectivized bundle 
$p^* \mathbb{P}(\pi_* \mathcal{L}) = \mathbb{P}(p^*  \pi_* \mathcal{L}) = \mathbb{P} (V\otimes \mathcal{O}_{\widetilde{S}})$ is the trivial projective
bundle --- see \cite[page 81]{mumford:1966}. This implies that $\pi_* \mathcal{L}$ is equipped with a flat projective
connection, whose monodromy is contained in the image of the group $\mathrm{Aut}^1(\mathcal{G}(n))$ of 
automorphisms of the Heisenberg group $\mathcal{G}(n)$ acting as the identity on the central subgroup $\mathbb{G}_m$. Note that 
$\mathrm{Aut}^1(\mathcal{G}(n))$ fits into the exact sequence
$$\begin{tikzcd} 0 \ar[r]& H(n) \ar[r]& \mathrm{Aut}^1(\mathcal{G}(n)) \ar[r] &
\mathrm{Aut}(H(n)) = \mathrm{Sp}(2g, \mathbb{Z}/n\mathbb{Z}) \ar[r]& 1\end{tikzcd}$$
and that there is a natural map 
$$\begin{tikzcd} \mathrm{Aut}^1(\mathcal{G}(n)) \ar[r]& \mathbb{P}\mathrm{GL}(V)\end{tikzcd}$$
extending the projective Schrödinger representation $H(n) \subset  \mathbb{P}\mathrm{GL}(V)$. Since
$ \mathrm{Aut}^1(\mathcal{G}(n))$ is a finite group, we obtain that the Mumford-Welters connection has finite
monodromy.

\subsection{Sums of theta bundles}
We first need a general fact about the direct sum of two projective connections.
\begin{proposition} \label{projconndirectsum}
Let $E_1,E_2$ be two vector bundles of rank $r_1,r_2$ equipped with two flat projective connections $\nabla_1, \nabla_2$. Assume that the
following relation holds
$$ (\det E_1)^{\otimes r_2} = (\det E_2)^{\otimes r_1}.$$
Then 
\begin{enumerate}[label=(\roman*)]
\item \label{firstsum} there exists a flat projective connection $\nabla$ on the direct sum $E = E_1 \oplus E_2$ preserving 
each subbundle $\mathbb{P}(E_i) \subset \mathbb{P}(E)$ and such that $\nabla_{|\mathbb{P}(E_i)} = \nabla_i$ for $i = 1,2$.
\item \label{secondsum} if both projective connections $\nabla_1$ and $\nabla_2$ have finite monodromy, then $\nabla$ also has finite
monodromy.
\end{enumerate}
\end{proposition}

\begin{proof}
\ref{firstsum} According to Lemma \ref{projistwisted}, a projective connection $\nabla_i$ on $E_i$ is equivalent
to a $\det (E_i)$-twisted connection on $E_i$ of weight $\frac{1}{r_i}$. But, by Lemma \ref{scale}, a $\det(E_1)$-twisted 
connection on $E_1$ of weight $\frac{1}{r_1}$ is equivalent to a $\det(E_1)^{\otimes r_2}$-twisted connection on $E_1$
of weight $\frac{1}{r_1r_2}$. If we denote $\lambda = (\det E_1)^{\otimes r_2} = (\det E_2)^{\otimes r_1}$ we obtain that
$E_1$ and $E_2$ are equipped with $\lambda$-twisted connections of weight $\frac{1}{r_1r_2}$. Then we apply
Lemma \ref{sum}.

\ref{secondsum} Let us denote by $G_i \subset \mathbb{P}\mathrm{GL}(r_i)$ the finite subgroups corresponding to the
images $\mathrm{im} \ \rho_i$. Then there exist finite \'etale Galois covers $\pi_i : X_i \rightarrow X$ with
Galois group $\mathrm{Gal}(X_i/X) = G_i$ such that the pull-back $(\pi_i^* E_i, \pi_i^* \nabla_i)$ is a vector bundle with a 
projective connection having trivial monodromy. Consider the fiber product
$$ \begin{tikzcd}\pi : Y = X_1 \times_X X_2 \ar[r]& X.\end{tikzcd} $$
This is a finite \'etale cover (not necessarily Galois) with the property that $(\pi^* E_i, \pi^* \nabla_i)$ has trivial
monodromy for $i=1,2$. We now apply Lemma \ref{trivialmonodromy} to the two vector bundles $\pi^* E_1$ and $\pi^* E_2$ :
there exists line bundles $L_1, L_2$ such that
$$ \pi^* E_1 = L_1^{\oplus r_1} \qquad \text{and} \qquad \pi^*E_2 = L_2^{\oplus r_2}$$
and we have the relation $(\pi^* \det E_1)^{\otimes r_2} =  (\pi^* \det E_2)^{\otimes r_1}$, which is equivalent
to $L_1^{\otimes r_1 r_2} = L_2^{\otimes r_1 r_2}$. Now we introduce the line bundle $\Lambda = L_2 L_1^{-1}$. Then
$\Lambda^{\otimes r_1r_2} = \mathcal{O}_Y$ and we consider the finite \'etale cover 
$\tilde{\pi} : \tilde{Y} \rightarrow Y$ of degree $r_1r_2$ corresponding to the $r_1r_2$-torsion line bundle $\Lambda$.
Then, by construction, $\tilde{\pi}^* \Lambda = \mathcal{O}_{\tilde{Y}}$. Hence, if we denote by 
$\alpha : \tilde{Y} \rightarrow X$ the composite map $\alpha = \pi \circ \tilde{\pi}$ we obtain that
$$ \alpha^* E_1 = L^{\oplus r_1} \qquad \text{and} \qquad \alpha^* E_2 = L^{\oplus r_2} $$
with $L = \tilde{\pi}^* L_1 =  \tilde{\pi}^* L_2$. Moreover the direct sum of the two trivial projective connections on $L^{\oplus r_1}$ and $L^{\oplus r_2}$ gives the trivial projective connection on $L^{\oplus r_1 + r_2} = \alpha^*( E_1 \oplus E_2)$.
Thus we obtain that the pull-back $\alpha^* \nabla$ is the trivial projective connection and we conclude by Lemma 
\ref{finitebecomestrivial} that
$\nabla$ has finite monodromy.
\end{proof}

We can now consider the flat projective Mumford-Welters connections (see subsection \ref{mumfordweltersconnection}) on the theta bundles $\pi_{J*}\Theta^k$ and $\pi_{P*}\Xi^k$.
We get
\begin{corollary}\label{combined-con}
The projective Mumford-Welters connections add up to a single flat projective connection on the direct sum over $\mathcal{M}_g$
\begin{equation}\label{thebundle}\pi_{J*}\Theta^k\oplus \pi_{g*}\left(\pi_{P*}\Xi^k\right).\end{equation}
This projective connection has finite monodromy.  Moreover, the morphisms (\ref{flatcombined}) from Proposition \ref{compositemorph} induce a projectively flat
morphism over $\mathcal{M}_g$
\begin{equation}\label{mainiso}\begin{tikzcd}
\left(\pi_{J*}\Theta^3\right)^{\vee} \oplus \pi_{g*} \left(\pi_{P*}\Xi^3\right)^{\vee} \ar[r]&  \mathbb{V}(\SL_2, 4). \end{tikzcd}
\end{equation}
\end{corollary}
\begin{proof} We begin by applying Lemma \ref{descend} to $\pi_g$, to obtain the connection on $\pi_{g*}\left(\pi_{P*}\Xi^k\right)$.
Proposition \ref{slope} gives us the equality on the determinants and ranks of the theta bundles, which enables us to apply Proposition \ref{projconndirectsum}, so that we obtain a single flat projective connection with finite monodromy on the bundle (\ref{thebundle}).
The flatness of (\ref{mainiso}) now follows from Proposition \ref{compositemorph}.
\end{proof}

\subsection{Parity} It was shown in \cite{oxbury.pauly:1996,pauly.ramanan:2001} that the morphism (\ref{mainiso}) becomes an isomorphism when restricted to the theta functions that are invariant for the involution which over the abelian torsors is given by
\begin{equation}\label{serreinv} \begin{tikzcd}\lambda \ar[r, maps to]& \lambda^{-1} \otimes K\end{tikzcd}\end{equation} 
on the Jacobian and Prym torsors.  We briefly show here that the corresponding subbundles of $\left(\pi_{J*}\Theta^3\right)^{\vee}\oplus \pi_{g*}\left(\pi_{P*}\Xi^3\right)^{\vee}$ are preserved by the projective flat connection.

All of the connections used by us can be understood to arise through a \emph{heat operator}.  We will use the approach to this introduced in \cite[\S 2.3]{vangeemen.dejong:1998}; see also \cite[\S 3.2--3.4]{BBMP:2020} (as this is the only section of this paper in which we need these notions, we just refer to the latter for background and notations).  In particular, let $\pi:\mathcal{M}\rightarrow S$ be a smooth surjective morphism of smooth schemes, and $L\rightarrow \mathcal{M}$ a line bundle such that $\pi_*L$ is locally free.  Assume we further have a projective heat operator $D:T_S\rightarrow (\pi_*\mathcal{W}_{\mathcal{M}/S}(L)\big/ \mathcal{O}_S$, with symbol $\rho:T_S\rightarrow \pi_*\operatorname{Sym}^2T_{\mathcal{M}/S}$.  

\begin{proposition}\label{isotypical}
    If a finite abelian group $\Gamma$ acts fiberwise on $\mathcal{M}$, and this action is linearised on $L$, then if the symbol $\rho$ is invariant under $\Gamma$, the projective connection induced by the heat operator $D$ will preserve the decomposition of $\pi_*L$ into isotypical components $$\pi_*L=\bigoplus_{\chi \in \widehat{\Gamma}} (\pi_*L)_{\chi},$$
    where $\widehat{\Gamma}$ is the character group of $\Gamma$.
\end{proposition}
\begin{proof}It suffices to show that the invariant subbundle $\left(\pi_{\ast}L\right)^{\Gamma}$ is preserved, the rest of the statement then follows after iterating over the changes of the linearization of the action of $\Gamma$ on $L$ by all possible characters.  This follows from the invariance of the projective heat operator that gives rise to the connection under the group action.  Given the invariance of the symbol, the invariance of the heat operator in turn is a consequence of its uniqueness in
realising the candidate symbol \cite[Theorem 3.4.1]{BBMP:2020}. Indeed, pulling back the heat operator under the map induced by any element of the group would give a new heat operator ---  if the symbol is invariant, it has to coincide with the given heat operator.
\end{proof}

We can now apply Proposition \ref{isotypical} to the family of abelian torsors $\pi_J: J^*\rightarrow \mathcal{M}_g$ and $\pi_P: P^{\rm even}\rightarrow \mathcal{R}_g$, where we consider the action of $\mathbb{Z}/2\mathbb{Z}$ generated by the involution (\ref{serreinv}) (using the relative canonical bundle for the family of curves).  In the case of abelian schemes, the projective Mumford-Welters connection was developed in the context we are using (by specifying a candidate symbol, and showing it satisfies the cohomological conditions needed to ensure it arises through a heat operator) in \cite[\S 2.3.8]{vangeemen.dejong:1998}, translating the arguments of \cite{welters:1983} from the deformation-theoretic description of connections.  The proof of this carries over essentially entirely to the case of abelian torsors that is of relevance for us here.  We just summarise the reasoning here for completeness.  

We need to verify that conditions (a), (b) and (c) of \cite[Theorem 3.4.1]{BBMP:2020} hold.  Condition (c) holds simply because the family of abelian torsors is proper, and condition (b) holds for families of abelian varieties \cite[p. 172]{oort-steenbrink:1980}, and hence also for torsors over them, as these are \'etale-locally isomorphic.  Since the relative canonical bundle for a family of abelian torsors is trivial, by \cite[Proposition 3.6.1]{BBMP:2020}, the morphism $\mu_L:\pi_{\ast}\operatorname{Sym}^2T_{\mathcal{M}/S}\rightarrow R^1\pi_{\ast}T_{\mathcal{M}/S}$ is just given by cupping with the relative Atiyah class of $L$ (cfr. \cite[Lemma 1.16 and (1.20)]{welters:1983}), and is an isomorphism \cite[\S 2.1]{welters:1983}.  Hence there is a unique symbol map $\rho_{\rm MW}$ that solves condition (a), given by $$\begin{tikzcd}\rho_{\rm MW}=-\mu_L^{-1}\circ \kappa_{\mathcal{M}/S}:T_S\ar[r]& \pi_*\operatorname{Sym}^2T_{\mathcal{M}/S}.\end{tikzcd}$$  This is the symbol map that gives rise to the Mumford-Welters connection.

In particular, we have 
\begin{lemma} The symbol $\rho_{\rm MW}$ for the Mumford-Welters connection is preserved by the involution (\ref{serreinv}).
\end{lemma}
\begin{proof}It suffices to remark that, as the action of $\Gamma$ is linearised, the Atiyah class of $L$ is invariant under $\Gamma$, and the Kodaira-Spencer morphism $\kappa_{\mathcal{M}/S}:T_S\rightarrow R^1\pi_{\ast}T_{\mathcal{M}/S}$ is necessarily invariant, as it is equivariant for general reasons, and $\Gamma$ acts trivially on $S$.
\end{proof}
We therefore have
\begin{corollary}\label{preserved}
The bundles $(\pi_{J*}\Theta^k)_+$ and $\left(\pi_{P*}\Xi^k\right)_+$ over $\mathcal{M}_g$ and $\mathcal{R}_g$ respectively, given by sections that are invariant under the involution (\ref{serreinv}) with its canonical linearization, are preserved by the Mumford-Welters connection.
\end{corollary}

\subsection{Conclusions}
We can now put everything together: Corollaries \ref{combined-con} and \ref{preserved}, together with the main theorem of \cite{oxbury.pauly:1996,pauly.ramanan:2001}, give
\mainresulta

Combining this with the finiteness from Corollary \ref{combined-con}, we finally have
\mainresultb

\def\cprime{$'$}

\end{document}